\RequirePackage{fix-cm}
\documentclass[a4paper,oneside,leqno,10pt,final]{article}

\usepackage{a4wide}
\usepackage[utf8]{inputenc}
\usepackage[margin=26mm,headheight=14pt]{geometry}

\usepackage{titlesec}

\usepackage{bm}

\usepackage{enumitem}
\setlist{itemsep=3pt,parsep=1pt,topsep=3pt,partopsep=1pt}
\setlist[enumerate,1]{label=(\arabic*)}
\setlist[enumerate,2]{label=(\alph*)}

\usepackage[dvipsnames]{xcolor}

\usepackage{graphicx}

\usepackage{mathtools}
\mathtoolsset{showonlyrefs}

\usepackage[only,llbracket,rrbracket]{stmaryrd}

\usepackage[T1]{fontenc}

\usepackage[p,osf]{scholax}
\usepackage{amsmath,amsthm}
\usepackage[scaled=1.075,ncf,vvarbb]{newtxmath}

\usepackage{chngcntr}

\usepackage{url}
\usepackage[colorlinks=true,linkcolor=blue,citecolor=Violet,linktocpage=false]{hyperref}
\newtheorem{theorem}{Theorem}[section]

\newtheorem{lemma}[theorem]{Lemma}
\newtheorem{corollary}[theorem]{Corollary}
\theoremstyle{definition}

\theoremstyle{remark}\newtheorem{remark}[theorem]{Remark}
\newcommand{\R}{\mathbb R}
\newcommand{\cH}{\mathcal H}
\newcommand{\cL}{\mathcal L}
\newcommand{\cC}{\mathcal C}
\newcommand{\cc}{\mathbf{c}}
\newcommand{\cY}{\mathcal Y}
\newcommand{\cW}{\mathcal W}

\newcommand{\dd}{\,\mathrm d}
\newcommand{\pp}{\mathbf{p}}
\newcommand{\np}{\bm{\xi}}

\newcommand{\qq}{\mathbf{q}}
\newcommand{\restr}{\mathbin{\vrule height 1.4ex depth -0.3ex width 0.07ex\vrule height 0.07ex depth 0ex width 0.75ex}}

\DeclareMathOperator{\Lip}{Lip}\DeclareMathOperator{\spt}{spt}
\DeclareMathOperator{\id}{id}

\DeclareMathOperator{\ap}{ap}

\DeclareMathOperator{\Der}{D}
\DeclareMathOperator{\trace}{trace}

\usepackage{microtype}
\usepackage{needspace}
\hypersetup{pdftitle={The optimal isoperimetric inequality for varifolds},
 pdfauthor={Mario Santilli}}
\newcommand{\Ciso}{\mathsf C_{d,N}}

\DeclareMathOperator{\dist}{dist}
\title{The optimal isoperimetric inequality for varifolds}
\author{Mario Santilli}

\begin{document}
\maketitle

\begin{abstract}
We prove the Brendle isoperimetric inequality for varifolds of finite
mass and finite total first variation, allowing an arbitrary singular
part. The distributional mean curvature is assumed only integrable.
The result applies to integral varifolds and, more generally,
to rectifiable varifolds with a positive lower density bound and weight
measure concentrated on countably many $C^2$ submanifolds. The proof
uses optimal transport and a locality theorem for the tangential
Laplacian of an ambient convex function. The isoperimetric constant is
sharp in codimensions one and two.
\end{abstract}

\section{Introduction}

We prove the isoperimetric inequality of Brendle for finite-mass
rectifiable $d$-varifolds $V=\mathbf v(S,\theta)$ in $\R^N$, $2\le d<N$,
with a positive lower density bound and weight measure
$\mu=\|V\|$ concentrated on countably many $C^2$ submanifolds.
We allow a generalized boundary: the first variation has the form
$\delta V=-H\mu+\eta\sigma$, where
$H\in L^1(\mu;\R^N)$,
$\sigma$ is a finite measure singular with respect to $\mu$,
$|\eta|=1$ $\sigma$-almost everywhere. No condition is imposed on the
support of $\sigma$.
Writing $M=\mu(\R^N)>0$ and $\theta\ge\theta_0>0$, our main result is
\[
 \sigma(\R^N)+\int|H|\dd\mu
 \ge \Ciso M^{-1/d}\int\theta^{1/d}\dd\mu
 \ge \Ciso\theta_0^{1/d}M^{(d-1)/d},
\]
where $\Ciso$ is given in \eqref{eq:Brendle-constant}.
In codimensions one and two, $\Ciso=d\omega_d^{1/d}$, with
$\omega_d$ the volume of the unit $d$-ball, and equality is attained by
flat round balls of constant multiplicity.
For integral varifolds, only finite mass and finite total first
variation are required; see Corollary~\ref{cor:integral-iso}.
Compactness of the support is not required.

Non-sharp isoperimetric and Sobolev inequalities for varifolds were
established by Allard~\cite[Section 7]{Allard1972} and
Michael--Simon~\cite{MichaelSimon1973}.
Almgren~\cite{Almgren1986} proved the sharp Euclidean isoperimetric
inequality for mass-minimizing integral currents in arbitrary dimension
and codimension.
Brendle~\cite{Brendle2021} proved a Sobolev inequality for smooth
submanifolds, with the sharp Euclidean constant in codimensions one and
two, and Brendle--Eichmair~\cite{BrendleEichmairTransport} gave an
optimal-transport proof of the corresponding isoperimetric inequality.
We extend the latter approach by identifying the absolutely continuous
part of the tangential Laplacian of an ambient convex function.
For real multiplicities, this identity includes the tangential
mean-curvature term $H\cdot\nabla^V u$.

\subsection*{Notation}

We define $\cc:\R^N\times\R^N\to\R$ by
$\cc(y,\xi)=y\cdot\xi$. For $\xi\in\R^N$ we write
$\cc_\xi(y)=y\cdot\xi$. The coordinate projections are
\[
 \pp(x,\omega)=x,\qquad \qq(x,\omega)=\omega.
\]
We use $B_r^N(a)=\{x:|x-a|<r\}$ for open balls and
$B^N=\{\xi:|\xi|\le1\}$ for the closed unit ball. Put
$\omega_k=\cL^k(B^k)$. If $ A, B : X \rightarrow X $ are linear operators on a finite dimensional inner product space $ X $, we set  $$A\bullet B=\trace(A^TB) \quad \text{and} \quad  |A| = (A\bullet A)^{1/2}. $$

\section{Optimal transport background}\label{sec:plans}
\subsection{Couplings and duality}

Suppose that $\Gamma\subset\R^N$ is nonempty and compact, $\alpha$ is a
probability measure supported on $\Gamma$, $K=B^N$, and $\beta$ is a
probability measure on $K$. Denote by $\Pi(\alpha,\beta)$ the probability
measures $\pi$ on $\Gamma\times K$ such that
\[
 \pp_\#\pi=\alpha,\qquad \qq_\#\pi=\beta.
\]
Such a measure is a \emph{coupling} with marginals $\alpha,\beta$.
For $(a,b)\in C(\Gamma)\times C(K)$ write
$(a\oplus b)(x,\xi)=a(x)+b(\xi)$. The marginal conditions give
\begin{equation}\label{eq:marginal-integral}
 \int_{\Gamma\times K}(a\oplus b)\dd\pi
 =\int_\Gamma a\dd\alpha+\int_K b\dd\beta.
\end{equation}
The set of couplings is nonempty, since it contains $\alpha\otimes\beta$.
The order of the factors is fixed: the eventual transport map goes from
$K$ to $\Gamma$, although the coupling is written on $\Gamma\times K$.

\begin{lemma}[cf.\ {\cite[Lemma 1.20]{AmbrosioGigli}}]\label{lem:graph}
Suppose $T:K\to\Gamma$ is Borel, $\pi\in\Pi(\alpha,\beta)$, and
$\pi\{(x,\xi):x=T(\xi)\}=1$. Then
\[
 \pi=(T,\id_K)_\#\beta,\qquad T_\#\beta=\alpha.
\]
\end{lemma}
\begin{proof}
Let $C=\{(x,\xi):x=T(\xi)\}$ and $Q=(T,\id_K)$. On $C$ one has
$Q\circ\qq=\id_C$. Hence, for a bounded Borel function $F$,
\[
 \int F\dd\pi=\int F(T(\xi),\xi)\dd\pi(x,\xi)
             =\int F(T(\xi),\xi)\dd\beta(\xi).
\]
This is the first assertion. Taking its first marginal proves the second.
\end{proof}

\begin{theorem}[Compact Kantorovich duality]\label{thm:duality}
For $k\in C(\Gamma\times K)$,
\begin{equation}\label{eq:duality}
 \max_{\pi\in\Pi(\alpha,\beta)}\int k\dd\pi
 =\inf_{\substack{(a,b)\in C(\Gamma)\times C(K)\\a\oplus b\ge k}}
       \left(\int a\dd\alpha+\int b\dd\beta\right).
\end{equation}
\end{theorem}
\begin{proof}
	See  \cite[Theorem~5.10(i)]{Villani}, applied to the cost
	$-k$; equivalently, or alternatively \cite[Theorem~5.1]{Edwards} on the compact
	metric spaces $\Gamma$ and $K$.
\end{proof}

\begin{theorem}[Convex dual potentials]\label{thm:dual}
For $k=\cc$ the dual infimum is attained. The minimizing pair can be
chosen as the restrictions of convex functions $u,\psi:\R^N\to\R$ with
$\Lip u\le1$ and $\Lip\psi\le R:=\max_{x\in\Gamma}|x|$, satisfying
\begin{align}
 \psi(\xi)&=\max_{x\in\Gamma}\{x\cdot\xi-u(x)\}
                  &&(\xi\in\R^N),\label{eq:psi-envelope}\\
 u(z)&=\max_{\xi\in K}\{z\cdot\xi-\psi(\xi)\}
                  &&(z\in\R^N).\label{eq:u-envelope}
\end{align}
Consequently,
\[
 \max_{\pi\in\Pi(\alpha,\beta)}\int\cc\dd\pi
 =\int_\Gamma u\dd\alpha+\int_K\psi\dd\beta.
\]
\end{theorem}
\begin{proof}
	See \cite[Theorem~5.10(iii), Definition~5.2 and
	Proposition~5.8]{Villani}, applied to the cost $-\cc$ on $\Gamma\times K$.
	The optimal potentials are chosen as a $(-\cc)$-conjugate pair, with
	signs adjusted to our convention. The envelopes
	\eqref{eq:psi-envelope}--\eqref{eq:u-envelope} give their ambient convex
	extensions and the stated Lipschitz bounds.
\end{proof}

\subsection{The optimal map}\label{sec:map}

Let $(u,\psi)$ be the optimal pair in Theorem~\ref{thm:dual}, let $\pi$
be an optimal coupling, and assume $\beta\ll\cL^N$. Define
\begin{equation}\label{eq:contact-def}
 \cC=\{(x,\xi)\in\Gamma\times K:u(x)+\psi(\xi)=x\cdot\xi\},
 \qquad
 \cC[E]=\{\xi\in K:\exists x\in E,\ (x,\xi)\in\cC\}.
\end{equation}
For Borel $E\subset\Gamma$, the projection $\cC[E]$ is analytic;
its measure is understood in the completion of $\beta$.
The relation $\cC$ is compact. The nonnegative duality gap has zero
$\pi$-integral, whence
\begin{equation}\label{eq:contact-plan}
 \pi(\cC)=1.
\end{equation}
The envelope formulas give
\begin{equation}\label{eq:subgradients}
 (x,\xi)\in\cC\quad\Longrightarrow\quad
 x\in\partial\psi(\xi),\qquad \xi\in\partial u(x).
\end{equation}
For instance, $\psi(\zeta)\ge x\cdot\zeta-u(x)
=\psi(\xi)+x\cdot(\zeta-\xi)$ for every $\zeta\in\R^N$.
Compactness in \eqref{eq:psi-envelope} also implies
\begin{equation}\label{eq:fibers-contact}
 \qq^{-1}(\xi)\cap\cC\ne\varnothing\qquad(\xi\in K).
\end{equation}

\begin{lemma}\label{lem:transport-map}
Let $D$ be the set of differentiability points of $\psi$ in $K$. There is
a Borel map $T:K\to\Gamma$ with $T=\nabla\psi$ on $D$, and
\begin{enumerate}
\item $\qq^{-1}(\xi)\cap\cC=\{(\nabla\psi(\xi),\xi)\}$ for every $\xi\in D$;
\item $\pi=(T,\id_K)_\#\beta$ and $T_\#\beta=\alpha$;
\item $\beta(\cC[E])=\alpha(E)$ for every Borel $E\subset\Gamma$.
\end{enumerate}
In particular the optimal coupling is unique.
\end{lemma}
\begin{proof}
At a differentiability point, $\partial\psi(\xi)=\{\nabla\psi(\xi)\}$.
Equations \eqref{eq:subgradients}--\eqref{eq:fibers-contact} prove (1),
and show that $\nabla\psi(\xi)\in\Gamma$. The differentiability set and
gradient are Borel; extend the latter to its complement using a fixed
point of $\Gamma$. Rademacher's theorem gives $\beta(K\setminus D)=0$.
The second marginal of $\pi$ therefore gives
$\pi(\Gamma\times(K\setminus D))=0$. Together with
\eqref{eq:contact-plan} and (1), this concentrates $\pi$ on the graph of
$T$, proving (2) by Lemma~\ref{lem:graph}.
For (3), outside $K\setminus D$ one has
$\cC[E]=T^{-1}(E)$; the pushforward identity finishes the proof.
Every optimal coupling is concentrated on this same graph, proving uniqueness.
\end{proof}

\begin{remark}\label{rem:contact-null}
If $\alpha(E)=0$, then $\beta(\cC[E])=0$. If the density of $\beta$ is
strictly positive almost everywhere on $K$, this also implies
$\cL^N(\cC[E])=0$. We will not require strict positivity: the annular
source densities used later vanish on an open set. The weighted
$\beta$-null assertion suffices. None of these assertions is a Lusin
property for the entire normal bundle of an arbitrary function.
\end{remark}

\section{Optimal transport and \texorpdfstring{$C^2$}{C2}-rectifiability}
\subsection{Semiconvexity on \texorpdfstring{$C^2$}{C2} submanifolds}\label{sec:semiconvex}

Suppose $1\le d<N$ and $\Sigma\subset\R^N$ is a $d$-dimensional
$C^2$ submanifold without boundary and not necessarily closed. Let
$T^\perp\Sigma$ be its vector normal bundle and $\nabla^\Sigma$ the
covariant derivative induced by the Euclidean metric. For $x\in\Sigma$,
write $\Pi_x^\Sigma$ for projection onto $T_x\Sigma$, and
$\xi^\perp=(I-\Pi_x^\Sigma)\xi$. The Weingarten map
$\cW_{(a,\omega)}\Sigma:T_a\Sigma\to T_a\Sigma$ is characterized by
\[
 \cW_{(a,\omega)}\Sigma(v)\cdot w
 =-\Der\widetilde\omega(a)(v)\cdot w,
\]
where $\widetilde\omega$ is a local normal field with value $\omega$ at
$a$. In particular,
\begin{equation}\label{eq:affine-hessian}
 \nabla^\Sigma_v(\nabla^\Sigma\cc_\xi)(x)\cdot w
 =\cW_{(x,\xi^\perp)}\Sigma(v)\cdot w.
\end{equation}
Let $\kappa_1(x,\omega)\le\cdots\le\kappa_d(x,\omega)$ be its
eigenvalues. They are continuous and satisfy
$\kappa_i(x,t\omega)=t\kappa_i(x,\omega)$ for $t\ge0$.

\begin{lemma}[Normal-bundle geometry]\label{lem:normal-geometry}
At $(x,\omega)\in T^\perp\Sigma$, choose an orthonormal eigenbasis
$\tau_i$ of $\cW_{(x,\omega)}\Sigma$ and an orthonormal basis
$\upsilon_1,\ldots,\upsilon_{N-d}$ of $T_x^\perp\Sigma$. Then
\[
 \{(\tau_i,-\kappa_i\tau_i):1\le i\le d\}
 \cup\{(0,\upsilon_\alpha):1\le\alpha\le N-d\}
\]
is a basis of $T_{(x,\omega)}(T^\perp\Sigma)$ and
\begin{equation}\label{eq:projection-jacobian}
 J^\Sigma(x,\omega) := J_d\pp(x,\omega)= 
 \prod_{i=1}^d(1+\kappa_i(x,\omega)^2)^{-1/2}.
\end{equation}
If $Z\subset\Sigma$ is Borel and $\cH^d(Z)=0$, then
$\cH^N(T^\perp\Sigma\cap\pp^{-1}(Z))=0$.
\end{lemma}
\begin{proof}
This follows from elementary differential-geometric considerations.
\end{proof}

Let now $u:\R^N\to\R$ be convex and Lipschitz. Its restriction to
$\Sigma$ is locally semiconvex for the induced Riemannian structure and its tangential gradient $ \nabla^\Sigma u $  belongs to $BV_{\rm loc}(\Sigma, \R^N)$ (cf.\ \cite{CannarsaSinestrari}). We denote by $ D $ the set of pointwise differentiability points of $ u $ relative to $ \Sigma $. It follows from \cite[3.83]{AFP} and \cite[3.1.8]{Fed69} that $ \nabla^\Sigma u $  is
$(\cH^d \llcorner \Sigma , d)$-approximately differentiable at $\cH^d$ almost every point of $ \Sigma $ and there exist pairwise disjoint Borel sets $P_1,P_2,\ldots\subset D$
such that
\begin{equation}\label{eq:P}
	P=\bigcup_{j=1}^\infty P_j,\qquad
	\cH^d(\Sigma\setminus P)=0,
	\qquad \Lip(\nabla^\Sigma u|_{P_j})<\infty.
\end{equation}
At such a
point we define the \emph{approximate covariant Hessian} of $ u $ as
\begin{equation}\label{eq:approx-hessian}
 \ap\nabla_\Sigma^2u(x)
 :=\Pi_x^\Sigma\circ\ap\Der (\nabla^\Sigma u)(x)|_{T_x\Sigma}:T_x\Sigma\to T_x\Sigma.
\end{equation}
It is a  symmetric  operator.

\begin{lemma}[Second-order expansion]\label{lem:second-order-expansion}
	There is a Borel set $G\subset P$, with $\cH^d(\Sigma\setminus G)=0$,
	such that, at every $a\in G$, $u|_\Sigma$ is differentiable and
	$\nabla^\Sigma u$ is approximately differentiable relative to $\Sigma$.
	For every $ a \in G $ and for every tangent-plane graph-representation
	$F(p)=a+p+f(p)$ of $\Sigma$ near $a$, where
	$f:W\subset T_a\Sigma\to T_a^\perp\Sigma$ is $C^2$,
	$0\in W$, and $f(0)=\Der f(0)=0$, one has
	\begin{equation}\label{eq:second-order-expansion}
		u(F(p))=u(a)+\nabla^\Sigma u(a)\cdot p
		+\tfrac12\ap\nabla^2_\Sigma u(a)(p)\cdot p+o(|p|^2).
	\end{equation}
	Moreover, there is a $C^2$ function $\widehat u_a$ on a relatively open
	neighborhood of $a$ in $\Sigma$ such that
	\[
	\widehat u_a(a)=u(a),\qquad
	\nabla^\Sigma\widehat u_a(a)=\nabla^\Sigma u(a),\qquad
	\nabla_\Sigma^2\widehat u_a(a)=\ap\nabla_\Sigma^2u(a).
	\]
	For every fixed $c\ge1$ there is a nonnegative, nondecreasing function
	$\omega_{a,c}$ such that $ \omega_{a,c}(r) \to 0 $ as $r\downarrow0$, and there is $ r_{a,c} > 0 $ such that
	\begin{equation}\label{eq:second-order-expansion tilde u}
		\sup_{q\in\Sigma\cap B_{cr}^N(a)}
	|u(q)-\widehat u_a(q)|\le\omega_{a,c}(r)r^2
	\qquad\text{for all $ 0 < r < r_{a,c} $\,.}
	\end{equation}
\end{lemma}
\begin{proof}This result follows from well known facts about semiconvex functions and we provide a sketch of the proof.
	
Put $A_a=\ap\nabla_\Sigma^2u(a)$. In a countable $C^2$ atlas,
semiconvexity and the second-order differentiability theorem, together
with its gradient estimate
\cite[Theorems 14.1 and 14.25(i$'$)]{Villani}, give quadratic polynomials
$Q_b$ satisfying
\[
 v(t)-Q_b(t)=o(|t-b|^2),\qquad
 Dv(t)-DQ_b(t)=o(|t-b|),\qquad v=u\circ\phi,
\]
where the second limit is through differentiability points; see also
\cite[Lemma 5]{BrendleEichmairTransport}. Choose a Borel full-measure
set $G\subset P$ where these conclusions hold.
For $a=\phi(b)\in G$, set $\widehat u_a=Q_b\circ\phi^{-1}$.
The first estimate gives $u-\widehat u_a=o(|\,\cdot-a|^2)$ on $\Sigma$.
The second, after changing coordinates, gives
$\nabla^\Sigma u-\nabla^\Sigma\widehat u_a=o(|\,\cdot-a|)$ through $D$;
comparison of approximate differentials therefore yields
$\nabla_\Sigma^2\widehat u_a(a)=A_a$.
This also gives the asserted value and gradient at $a$.
In any tangent-plane graph centered at $a$, the Christoffel symbols
vanish at the origin. Taylor expansion of $\widehat u_a$ in that graph
therefore gives the asserted second-order expansion, with the same $G$.

Choose $\delta_a>0$ so that $\Sigma\cap B_{2\delta_a}^N(a)$ lies in
the domain of $\widehat u_a$ and, for $0<t\le\delta_a$, the function
\[
 \varepsilon_a(t)=
 \sup_{\substack{q\in\Sigma\\0<|q-a|\le t}}
 \frac{|u(q)-\widehat u_a(q)|}{|q-a|^2}
\]
is finite. It is nondecreasing and tends to zero at zero. Set
\[
 r_{a,c}=\delta_a/c,\qquad
 \omega_{a,c}(r)=c^2\varepsilon_a(\min\{cr,\delta_a\}).
\]
For $0<r<r_{a,c}$ and $q\in\Sigma\cap B_{cr}^N(a)$,
$|u(q)-\widehat u_a(q)|\le\varepsilon_a(cr)c^2r^2$.
This proves \eqref{eq:second-order-expansion tilde u};
$\widehat u_a$ and $\delta_a$ are fixed before choosing $c$ and $r$.
\end{proof}

Fix now a Borel map $  \Phi:T^\perp\Sigma\to\R^N $ such that
\begin{equation}\label{eq:Phi}
\Phi(x,\omega)=\nabla^\Sigma u(x)+\omega  \quad \text{whenever $ u $ is differentiable at  $ x $ relative to $ \Sigma $\,.}
\end{equation}
Notice that $ \cH^N(\{(x, \omega)\in T^\perp\Sigma: \text{$ u $ is not differentiable at $ x $ relative to $ \Sigma $} \}) =0 $. Moreover, we observe that at $\cH^N$ almost every point $(x,\omega) \in T^\perp \Sigma $  the map $\Phi$ is $(\cH^N \llcorner T^\perp \Sigma, N)$-approximately differentiable and
\[
 \ap\Der\Phi(x,\omega)(\tau,\upsilon)
 =\ap\Der (\nabla^\Sigma u)(x)(\tau)+\upsilon \quad \text{for $ (\tau, \upsilon) \in T_{(x, \omega)}T^\perp \Sigma $\,.}
\]

\begin{lemma}[Normal-fiber area formula]\label{lem:coarea}
Let $E\subset T^\perp\Sigma$ be Borel and $F\ge0$ be Borel on
$T^\perp\Sigma$. Then
\begin{align}
 &\int_{\Sigma}\int_{\{\omega:(x,\omega)\in E\}}
 \big|\det[\ap\nabla^2_\Sigma u(x)-\cW_{(x,\omega)}\Sigma]\big|
 F(x,\omega)\dd\cH^{N-d}(\omega)\dd\cH^d(x)\notag\\
 &\hspace{12mm}=
 \int_{\R^N}\sum_{(x,\omega)\in\Phi^{-1}(\xi)\cap E\cap\pp^{-1}(P)}
 F(x,\omega)\dd\cL^N(\xi).
 \label{eq:fiber-area}
\end{align}
The left integrand is defined to be zero at base points outside $P$.
Both sides are nonnegative extended integrals.
\end{lemma}
\begin{proof}
On each $T^\perp\Sigma\cap\pp^{-1}(P_j)$, the map $\Phi$ is Lipschitz.
At almost every $(x,\omega)$ there, put
$A=\ap\nabla_\Sigma^2u(x)$, $W=\cW_{(x,\omega)}\Sigma$, and
$ Q=\ap\Der\Phi(x,\omega)$. Use the orthogonal basis
$h_i=(\tau_i,-\kappa_i\tau_i)$, $v_\alpha=(0,\upsilon_\alpha)$
from Lemma~\ref{lem:normal-geometry}, and put $q=N-d$.
The differential formula preceding the lemma gives
\[
 Q h_i=(A-W)\tau_i+n_i,\quad n_i\in T_x^\perp\Sigma,
 \qquad Q v_\alpha=\upsilon_\alpha.
\]
Let $\zeta=\upsilon_1\wedge\cdots\wedge\upsilon_q$.
Since $n_i\wedge\zeta=0$, multilinearity and alternation give
\[
 (Q h_1\wedge\cdots\wedge Q h_d)\wedge\zeta
 =\det(A-W)(\tau_1\wedge\cdots\wedge\tau_d)\wedge\zeta.
\]
The norm of this $N$-vector is $|\det(A-W)|$, whereas the
$N$-volume of the domain basis is
$\prod_{i=1}^d(1+\kappa_i^2)^{1/2}=1/J^\Sigma(x,\omega)$.
Consequently,
\[
 \ap J_N\Phi(x,\omega)
 =|\det(A-W)|J^\Sigma(x,\omega).
\]
Apply the weighted area formula on the disjoint Lipschitz restrictions
and sum, then apply coarea to $\pp$, whose Jacobian is $J^\Sigma$
and whose fibers are isometric to $T_x^\perp\Sigma$;
see \cite[3.2.3 and 3.2.22]{Fed69}. This proves the asserted identity.
The excluded set over $\Sigma\setminus P$ is null by
Lemma~\ref{lem:normal-geometry}; its image is not used.
\end{proof}

\subsection{\texorpdfstring{$C^2$}{C2}-rectifiable sets}

Let $1\le d<N$, let $\Gamma\subset\R^N$ be compact, and let $S\subset\Gamma$
be Borel and carried up to $\cH^d$-null sets by countably many
$d$-dimensional $C^2$ submanifolds $\Sigma_j$. Let $\theta:S\to(0,\infty)$
be Borel, finite almost everywhere, with
\[
 0<M:=\int_S\theta\dd\cH^d<\infty,
 \qquad \mu=\theta\cH^d\restr S,
 \qquad \alpha=\mu/M, \qquad \Gamma=\spt\mu\,.
\]

Let $\rho:\R^N\to[0,\infty)$ be Borel, vanish outside $B^N$, and satisfy
$\int\rho\dd\cL^N=1$. Set $\beta=\rho\cL^N$ and choose the optimal pair
$(u,\psi)$ of Theorem~\ref{thm:dual}. Choose pairwise disjoint Borel
sets $G_j\subset S\cap\Sigma_j$ from the full-measure sets of
Lemma~\ref{lem:second-order-expansion}, so that
$G=\bigcup_jG_j$ has full $\mu$-measure and the tangent plane of $\mu$
agrees with $T_x\Sigma_j$ at every $x\in G_j$.
Write $\Pi_x^S$ for the orthogonal projection onto this plane. For
$x\in G_j$ put
\begin{align*}
 T_xS&=T_x\Sigma_j,&
 \nabla^S u(x)&=\nabla^{\Sigma_j}u(x),&
 \ap\nabla_S^2u(x)&=\ap\nabla_{\Sigma_j}^2u(x),\\
 \cW_{(x,\omega)}S&=\cW_{(x,\omega)}\Sigma_j,&
 H_S(x,\omega)&=\trace\cW_{(x,\omega)}\Sigma_j,&
 \Phi(x,\omega)&=\nabla^S u(x)+\omega.
\end{align*}
Define
\begin{equation}\label{eq:active-normal}
 \cY_x=\{\omega\in T_x^\perp S:(x,\Phi(x,\omega))\in\cC\}
 \quad(x\in G),
 \qquad
 \cY=\{(x,\omega):x\in G,\ \omega\in\cY_x\}.
\end{equation}
All these sets are Borel on the countably many bundle charts.

\begin{lemma}\label{lem:contact-relation}
If $\xi\in B^N$ is a differentiability point of $\psi$ and
$\Phi^{-1}(\xi)\cap\cY\ne\varnothing$, then
\[
 \Phi^{-1}(\xi)\cap\cY
 =\{(\nabla\psi(\xi),\xi-\nabla^S u(\nabla\psi(\xi)))\}.
\]
Moreover, for every $(x,\omega)\in\cY$,
\begin{equation}\label{eq:contact-positive}
 \ap\nabla_S^2u(x)-\cW_{(x,\omega)}S\ge0,
 \qquad
 |\nabla^S u(x)|^2+|\omega|^2\le1.
\end{equation}
For every $x\in G$, every contact slope $\xi$ at $x$ has
$\xi=\nabla^S u(x)+\xi^\perp$, with $(x,\xi^\perp)\in\cY$.
\end{lemma}
\begin{proof}
The unique-contact assertion follows from Lemma~\ref{lem:transport-map}(1);
once the base point is fixed, the normal vector is uniquely
$\xi-\nabla^S u(x)$. For the remaining assertions use
$\xi\in\partial u(x)$ from \eqref{eq:subgradients}. At $x\in G_j$, write
$\Sigma_j$ in tangent-plane graph coordinates. Lemma
\ref{lem:second-order-expansion} and the $C^2$ expansion of the graph give
\begin{align*}
 0&\le u(F(z))-u(x)-\xi\cdot(F(z)-x)\\
 &=\bigl(\nabla^S u(x)-\Pi_x^S\xi\bigr)\cdot z
 +\frac12\bigl(\ap\nabla_S^2u(x)
       -\cW_{(x,\xi^\perp)}S\bigr)[z,z]+o(|z|^2).
\end{align*}
Testing both signs of $z$ first cancels the linear term and then proves
positivity of the quadratic term. The length bound follows from
$|\xi|\le1$ and orthogonality.
\end{proof}

\begin{lemma}\label{lem:fiber-balance}
For $\mu$ almost every $x\in S$,
\begin{equation}\label{eq:fiber-balance}
 \frac{\theta(x)}{M}
 =\int_{\cY_x}
 \det[\ap\nabla_S^2u(x)-\cW_{(x,\omega)}S]
 \rho(\Phi(x,\omega))\dd\cH^{N-d}(\omega).
\end{equation}
\end{lemma}
\begin{proof}
Denote the nonnegative fiber integral on the right by $J(x)$, and put it
zero off $G$. It is measurable by local bundle coordinates and Tonelli's
theorem. 
For a Borel $E\subset G$, apply Lemma~\ref{lem:coarea} on each disjoint
$G_j$ to the domain $\cY\cap\pp^{-1}(E\cap G_j)$ and weight
$\rho\circ\Phi$. Summing gives
\begin{equation}\label{eq:integrated-balance}
 \int_E J\dd\cH^d
 =\int_{B^N}\rho(\xi)
   \cH^0(\Phi^{-1}(\xi)\cap\cY\cap\pp^{-1}(E))\dd\xi.
\end{equation}
For every differentiability point of $\psi$, the multiplicity in this
integral is $1_{\cC[E]}(\xi)$ by Lemma~\ref{lem:contact-relation}. The
other slopes form a Lebesgue-null set. Hence
\[
 \int_EJ\dd\cH^d=\beta(\cC[E])=\alpha(E)
                    =\int_E\frac{\theta}{M}\dd\cH^d.
\]
Equality of these measures on the base proves \eqref{eq:fiber-balance}
by Radon--Nikodym uniqueness. Exceptional bases in $\Gamma\setminus G$
have zero $\alpha$-mass and hence zero active $\beta$-image.
\end{proof}

\section{Comparison with a \texorpdfstring{$C^2$}{C2} submanifold}
\label{sec:boundary-comparison}

Throughout this section $\Omega\subset\R^N$ is open, $2\le d<N$, and
$V=\mathbf v(S,\theta)$ is a rectifiable $d$-varifold with locally
bounded first variation, $\mu=\|V\|=\theta\cH^d\restr S$, and
$\theta\ge\theta_0>0$ $\mu$-almost everywhere.
Write $\Pi_x=\Pi_x^V$ for the orthogonal projection onto $T_xV$,
$\lambda=\|\delta V\|$, and $p=d/(d-1)$. The Lebesgue decomposition
of the first variation is written as
\begin{equation}\label{eq:first-variation-boundary}
 \delta V(X)=\int\Pi_x\bullet\Der X(x)\dd\mu(x)
 =-\int H\cdot X\dd\mu+\int\eta\cdot X\dd\sigma
 \qquad\bigl(X\in C_c^1(\Omega;\R^N)\bigr),
\end{equation}
where $H\in L^1_{\rm loc}(\mu;\R^N)$, $\sigma$ is a positive Radon
measure singular with respect to $\mu$, and $\eta$ is Borel with
$|\eta|=1$ $\sigma$-almost everywhere. Thus
$\lambda=|H|\mu+\sigma$; no condition on the support of $\sigma$ is
imposed. For a vector-valued $v\in C^1(\Omega;\R^N)$ set
$\Der^Vv(x)=\Der v(x)\circ\Pi_x$; for a scalar $h\in C^1(\Omega)$,
set $\nabla^Vh=\Pi\nabla h$.

For an embedded $C^2$ submanifold $\Sigma\subset\Omega$, denote its
nearest-point projection, where uniquely defined, by $\np_\Sigma$,
and put $\nu_\Sigma(z)=z-\np_\Sigma(z)$. Near every $a\in\Sigma$,
there is an open neighborhood $U\Subset\Omega$ on which
$\np_\Sigma$ is single-valued and $C^1$, with
\[
 \Der\np_\Sigma(a)=\Pi_a^\Sigma.
\]
The neighborhood may be chosen so that $\np_\Sigma(U)$ lies in a
compact subset of $\Sigma$. The function
$h=|\nu_\Sigma|=\dist(\,\cdot\,,\Sigma)$ is $C^2$ on $U\setminus\Sigma$;
see \cite[Theorems 2 and C]{LeobacherSteinicke2021}.

We use the non-sharp Sobolev inequality with the full first-variation
measure:
\begin{equation}\label{eq:full-sobolev}
 \|h\|_{L^p(\mu)}
 \le C_S\left(\int|\nabla^Vh|\dd\mu+\int h\dd\lambda\right)
 \qquad(0\le h\in C_c^1(\Omega)).
\end{equation}
Here $C_S$ depends only on $d,N,\theta_0$.
This is \cite[Theorem~7.1]{Allard1972}; see also
\cite[Theorem~10.1(2a)]{MenneWeak2016}. 

\begin{lemma}[A truncation estimate]\label{lem:scalar-boundary-estimate}
Let $v:\Omega\to[0,\infty)$ be Lipschitz with compact support in
$\Omega$, and suppose $v\in C^2(E)$, where $E=\{v>0\}$.
Set $\gamma=|\Pi\nabla v|$ on $E$ and $\gamma=0$ off $E$, and put
\[
 m=\mu(E),\qquad b=\lambda(E),\qquad
 L_v=\|\gamma\|_{L^\infty(\mu)}.
\]
Suppose $M_v\ge0$ is finite and
\begin{equation}\label{eq:scalar-truncation-energy}
 \frac1\varepsilon\int_{\{t<v<t+\varepsilon\}}\gamma^2\dd\mu
 \le M_v\qquad(t>0,\ \varepsilon>0).
\end{equation}
Then, for $C=C(d,N,\theta_0)$,
\begin{equation}\label{eq:scalar-boundary-estimate}
 \int\gamma\dd\mu
 \le C m^{1/d}M_v+C L_v b^{d/(d-1)}.
\end{equation}
\end{lemma}
\begin{proof}
The tangential derivative of $v$ is zero $\mu$-almost everywhere on
$\{v=0\}$, so $\gamma=|\nabla^Vv|$ almost everywhere.
Put $m(t)=\mu(\{v>t\})$ for $t>0$. On
$Z_t=S\cap\{v=t,\ \gamma>0\}$ define
\[
 P(t)=\int_{Z_t}\theta\dd\cH^{d-1},\qquad
 Q(t)=\int_{Z_t}\gamma\theta\dd\cH^{d-1},\qquad
 J(t)=\int_{Z_t}\frac{\theta}{\gamma}\dd\cH^{d-1}.
\]
Coarea on $S$ gives
\begin{equation}\label{eq:scalar-level-coarea}
 \int_0^\infty P(t)\dd t=\int\gamma\dd\mu,
 \qquad J(t)\le-m'(t)\quad\text{for almost every }t>0.
\end{equation}
Indeed, $J$ is the density of the pushforward of
$\mu\restr\{v>0,\ \gamma>0\}$ under $v$; the remaining contribution
to the distributional measure $-Dm$ is nonnegative.
Coarea and \eqref{eq:scalar-truncation-energy} also yield
$Q(t)\le M_v$ for almost every $t>0$.

Apply \eqref{eq:full-sobolev} to
$T_{t,\varepsilon}(v)$, where
\[
 T_{t,\varepsilon}(s)=\min\{(s-t)_+/\varepsilon,1\}.
\]
Smooth approximations of this scalar function give admissible
$C_c^1$ tests, since $t>0$ and $v$ is $C^2$ on $E$.
The limit in the gradient term is valid because tangential gradients
vanish almost everywhere on every fixed level set.
Letting $\varepsilon\downarrow0$ and using coarea gives
\begin{equation}\label{eq:scalar-level-sobolev}
 m(t)^{(d-1)/d}\le C_S(P(t)+b)
 \quad\text{for almost every }t>0.
\end{equation}
Cauchy--Schwarz on $Z_t$ and \eqref{eq:scalar-level-coarea} imply
\begin{equation}\label{eq:scalar-flux-bound}
 P(t)^2\le Q(t)J(t)\le M_v[-m'(t)]
 \quad\text{for almost every }t>0.
\end{equation}

Set $b_*=(2C_Sb)^p$. For $m(t)>b_*$,
\eqref{eq:scalar-level-sobolev} gives
$P(t)\ge(2C_S)^{-1}m(t)^{(d-1)/d}$, so
\[
 P(t)\le2C_SM_v[-m'(t)]m(t)^{-(d-1)/d}.
\]
The total decrease of the nonincreasing function $m^{1/d}$ therefore
bounds the integral over these levels by $2dC_SM_vm^{1/d}$.
The complementary levels form a final interval. By coarea and
$\gamma\le L_v$, their contribution is at most $L_vb_*$: take levels
$t$ with $m(t)\le b_*$ and let them decrease to the initial endpoint
of that interval. Empty intervals and the case $m\le b_*$ are
included. Combining the two estimates with
\eqref{eq:scalar-level-coarea} proves the assertion.
\end{proof}

\begin{lemma}[$L^1$ comparison with a $C^2$ submanifold]
\label{lem:L1-boundary-comparison}
Assume \eqref{eq:first-variation-boundary}, with
$H\in L^1_{\rm loc}(\mu)$ and $\sigma\perp\mu$, and let
$\Sigma\subset\Omega$ be an embedded $d$-dimensional $C^2$
submanifold. Then, for $\mu$-almost every $a\in\Sigma$,
\begin{align}
 \int_{B_r^N(a)}|\nu_\Sigma|\dd\mu&=o(r^{d+2}),
 \label{eq:boundary-L1-height}\\
 \int_{B_r^N(a)}|\Pi_z^V-\Pi^\Sigma_{\np_\Sigma(z)}|\dd\mu(z)
 &=o(r^{d+1}).\label{eq:boundary-L1-tilt}
\end{align}
\end{lemma}
\begin{proof}
Choose $a\in\Sigma$ with positive finite $d$-density and such that,
writing $B_r=B_r^N(a)$,
\[
 m_0(r):=\mu(B_r\setminus\Sigma)=o(r^d),\qquad
 b_0(r):=\int_{B_r\setminus\Sigma}(1+|H|)\dd\mu+\sigma(B_r)
 =o(r^d).
\]
Since $H\in L^1_{\rm loc}(\mu)$, the measure $(1+|H|)\mu$ is
locally finite. Differentiating its restriction to
$\Omega\setminus\Sigma$, and using $\sigma\perp\mu$, gives
$b_0(r)/\mu(B_r)\to0$ at $\mu$-almost every $a\in\Sigma$.
The same argument for $\mu\restr(\Omega\setminus\Sigma)$ gives
$m_0(r)/\mu(B_r)\to0$. The positive finite $d$-density then gives
the displayed decay rates. Moreover, $T_zV=T_z\Sigma$ for
$\mu$-almost every $z\in\Sigma$.
Fix a sufficiently small tubular neighborhood $U$ of $a$; all radii
below satisfy $B_{4r}\Subset U$. Constants may depend on this fixed
neighborhood but not on $r$ or on the truncation parameters.

\medskip\noindent\emph{Step 1: the distance and the tilt.}
On $U\setminus\Sigma$ put
\[
 h=|\nu_\Sigma|,\qquad n=\nu_\Sigma/h,\qquad
 \widehat\Pi_z=\Pi^\Sigma_{\np_\Sigma(z)}.
\]
Then $\nabla h=n$. The projection differential formula in
\cite[Theorem C]{LeobacherSteinicke2021}, with the bounded second
fundamental form on the fixed neighborhood, gives
\[
 D\nu_\Sigma=I-\widehat\Pi+\mathcal E,\qquad |\mathcal E|\le Ch.
\]
Here the bound follows by shrinking $U$ so that the inverse in that
formula is uniformly bounded. Differentiating $n=\nu_\Sigma/h$ yields
\[
 D^2h=\frac{I-\widehat\Pi-n\otimes n}{h}
             +\frac{\mathcal E}{h}.
\]
Define, at $\mu$-almost every point off $\Sigma$,
\[
 k=|\Pi n|=|\nabla^Vh|,\qquad
 \alpha^2=\Pi\bullet(I-\widehat\Pi-n\otimes n),\qquad \alpha\ge0.
\]
Since $n\in T_{\np_\Sigma(z)}^\perp\Sigma$, the operator
$I-\widehat\Pi-n\otimes n$ is an orthogonal projection. Thus
\begin{equation}\label{eq:distance-tilt-decomposition}
 \tfrac12|\Pi-\widehat\Pi|^2=k^2+\alpha^2,\qquad
 \Pi\bullet D^2h\ge\frac{\alpha^2}{h}-C.
\end{equation}
Set $k=\alpha=0$ on $\Sigma$; the first identity then holds
$\mu$-almost everywhere in $U$. Put
\[
 I(r)=\int_{B_r}k\dd\mu,\qquad
 J_0(r)=\int_{B_r}h\dd\mu,\qquad
 K(r)=I(r)+r^{-1}J_0(r).
\]

\medskip\noindent\emph{Step 2: truncated first-variation tests.}
Take $\chi_r\in C_c^\infty(B_{2r})$, with $0\le\chi_r\le1$,
$\chi_r=1$ on $B_r$, $|D\chi_r|\le C/r$, and
$|D^2\chi_r|\le C/r^2$. Set $v_r=\chi_rh$ and extend it by zero.
It is Lipschitz with compact support, is $C^2$ on
$E_r=\{v_r>0\}\subset B_{2r}\setminus\Sigma$, and satisfies
$|Dv_r|\le C$ there, since $h(z)\le|z-a|$.
Write $\gamma_r=|\nabla^Vv_r|$, with value zero off $E_r$.
By the product rule and \eqref{eq:distance-tilt-decomposition},
\[
 \Pi\bullet D^2v_r
 \ge\chi_r\frac{\alpha^2}{h}
       -C\left(1+r^{-1}k+r^{-2}h\right)
 \qquad\text{on }E_r.
\]

Let $T:\R\to[0,1]$ be smooth and nondecreasing, and vanish on
$(-\infty,t_0]$ for some $t_0>0$. The field
$X=T(v_r)\nabla v_r$, extended by zero, belongs to
$C_c^1(\Omega;\R^N)$: its support is a compact subset of $E_r$.
The first variation gives
\begin{align*}
 \int T'(v_r)\gamma_r^2\dd\mu
 +\int_{E_r}T(v_r)\Pi\bullet D^2v_r\dd\mu
 &=-\int T(v_r)H\cdot Dv_r\dd\mu
   +\int T(v_r)\eta\cdot Dv_r\dd\sigma.
\end{align*}
Consequently,
\begin{equation}\label{eq:distance-truncated-energy}
 \int T'(v_r)\gamma_r^2\dd\mu
 +\int_{E_r}T(v_r)\chi_r\frac{\alpha^2}{h}\dd\mu
 \le\mathcal M(r),
\end{equation}
where we fix
\begin{equation}\label{eq:distance-energy-size}
 \mathcal M(r)=C\left[b_0(2r)+r^{-1}I(2r)+r^{-2}J_0(2r)\right].
\end{equation}
Indeed, $|T(v_r)Dv_r|\le C$ and this field is supported off
$\Sigma$, so the curvature and boundary terms satisfy
\[
 \left|\int T(v_r)H\cdot Dv_r\dd\mu\right|
 +\left|\int T(v_r)\eta\cdot Dv_r\dd\sigma\right|
 \le C\left[\int_{B_{2r}\setminus\Sigma}|H|\dd\mu
                   +\sigma(B_{2r})\right].
\]
Thus local $L^1$ integrability of $H$ suffices for the first-variation
estimate, uniformly in the truncation parameters.

Approximate $T_{t,\varepsilon}(s)=\min\{(s-t)_+/\varepsilon,1\}$
by smooth nondecreasing truncations in
\eqref{eq:distance-truncated-energy}, and drop its second term.
The limit is justified by the vanishing of tangential gradients on
fixed level sets, and gives
\[
 \frac1\varepsilon\int_{\{t<v_r<t+\varepsilon\}}\gamma_r^2\dd\mu
 \le\mathcal M(r)\qquad(t,\varepsilon>0).
\]
Alternatively, take smooth $T$ tending increasingly to one on
$(0,\infty)$ and drop the first term. Monotone convergence gives
\begin{equation}\label{eq:distance-angular-bound}
 \int_{B_r\setminus\Sigma}\frac{\alpha^2}{h}\dd\mu
 \le\mathcal M(r).
\end{equation}
All vector-field tests vanish near $\Sigma$.

\medskip\noindent\emph{Step 3: the Sobolev estimate and iteration.}
Set $B(r)=\int\gamma_r\dd\mu$. Apply
Lemma~\ref{lem:scalar-boundary-estimate} to $v_r$, using
$\mu(E_r)\le m_0(2r)$, $\lambda(E_r)\le b_0(2r)$,
$\|\gamma_r\|_\infty\le C$, and the preceding energy estimate. Then
\begin{equation}\label{eq:boundary-B-bound}
 B(r)\le C m_0(2r)^{1/d}\mathcal M(r)+C b_0(2r)^p.
\end{equation}
On $B_r$, $v_r=h$, so $I(r)\le B(r)$. Moreover, H\"older's
inequality and \eqref{eq:full-sobolev} give
\begin{equation}\label{eq:boundary-J-bound}
 J_0(r)\le\int v_r\dd\mu
 \le C m_0(2r)^{1/d}\left[B(r)+r b_0(2r)\right].
\end{equation}
To apply the Sobolev inequality here, approximate $v_r$ by smooth
scalar truncations vanishing near its zero set. They converge
uniformly to $v_r$, with uniformly bounded gradients converging
tangentially $\mu$-almost everywhere. The boundary integrals converge
as well. We also used $0\le v_r\le2r$ and
$\lambda(E_r)\le b_0(2r)$.

To verify the scale of the error terms, put
$a_r=m_0(2r)^{1/d}$ and $b_r=b_0(2r)$.
Since $r^{-1}I(2r)+r^{-2}J_0(2r)\le2r^{-1}K(2r)$, combining
\eqref{eq:distance-energy-size}--\eqref{eq:boundary-J-bound} gives
\[
 \begin{split}
 K(r)\le{}&C\frac{a_r}{r}\left(1+\frac{a_r}{r}\right)K(2r)\\
           &+C\left(1+\frac{a_r}{r}\right)(a_rb_r+b_r^p).
 \end{split}
\]
Now $a_r=o(r)$, $b_r=o(r^d)$, and
$dp=d^2/(d-1)>d+1$. Hence $a_rb_r+b_r^p=o(r^{d+1})$ and
\[
 K(r)\le\varepsilon(r)K(2r)+o(r^{d+1}),\qquad
 \varepsilon(r)\longrightarrow0.
\]
For $A(r)=r^{-d-1}K(r)$, reduce the fixed radius so that
$2^{d+1}\varepsilon(r)\le1/2$. Thus
\[
 A(r)\le\omega(r)+\tfrac12 A(2r),\qquad \omega(r)\to0.
\]
Fix $\varepsilon>0$ and choose $r_\varepsilon>0$ so that the
recurrence holds and $\omega(t)\le\varepsilon$ for
$0<t<r_\varepsilon$. If $2^kr\in[r_\varepsilon,2r_\varepsilon)$,
iteration gives
\[
 A(r)\le2\varepsilon+
 2^{-k}\sup_{r_\varepsilon\le t<2r_\varepsilon}A(t).
\]
This supremum is finite, since $I$ and $J_0$ are finite and
nondecreasing. Letting $r\downarrow0$ and then
$\varepsilon\downarrow0$ proves
\begin{equation}\label{eq:distance-height-gradient-decay}
 I(r)=o(r^{d+1}),\qquad J_0(r)=o(r^{d+2}).
\end{equation}
The second estimate is \eqref{eq:boundary-L1-height}.

\medskip\noindent\emph{Step 4: recovery of the full tilt.}
Equations \eqref{eq:distance-energy-size} and
\eqref{eq:distance-height-gradient-decay} imply
$\mathcal M(r)=o(r^d)$. Hence \eqref{eq:distance-angular-bound}
and Cauchy--Schwarz give
\[
 \int_{B_r}\alpha\dd\mu
 \le\left(\int_{B_r\setminus\Sigma}\frac{\alpha^2}{h}\dd\mu\right)^{1/2}
     \left(\int_{B_r}h\dd\mu\right)^{1/2}
 =o(r^{d+1}).
\]
By \eqref{eq:distance-tilt-decomposition},
\[
 \int_{B_r}|\Pi-\widehat\Pi|\dd\mu
 \le\sqrt2\left[I(r)+\int_{B_r}\alpha\dd\mu\right]
 =o(r^{d+1}),
\]
which is \eqref{eq:boundary-L1-tilt}. In codimension one
$I-\widehat\Pi=n\otimes n$, so $\alpha=0$ and the estimate for
$I(r)$ already controls the full tilt.
\end{proof}

\section{The tangential Laplacian of an ambient convex function}\label{sec:laplacian}

Let $\Omega\subset\R^N$ be open, $2\le d<N$, and let
$V=\mathbf v(S,\theta)$ be rectifiable, with
$\mu=\|V\|=\theta\cH^d\restr S$ and $\theta\ge\theta_0>0$
$\mu$-almost everywhere. Assume the generalized first-variation
identity \eqref{eq:first-variation-boundary}, namely
\[
 \delta V(X)=-\int H\cdot X\dd\mu+\int\eta\cdot X\dd\sigma
 \qquad(X\in C_c^1(\Omega;\R^N)),
\]
where $H\in L^1_{\rm loc}(\mu;\R^N)$, $\sigma$ is a positive Radon
measure singular with respect to $\mu$, and $|\eta|=1$
$\sigma$-almost everywhere. We use the notation of Section~\ref{sec:boundary-comparison}.

Let $\Sigma_j\subset\Omega$ be $d$-dimensional $C^2$ submanifolds
whose union has full $\mu$-measure, and fix a convex $1$-Lipschitz
function $u:\R^N\to\R$. For each $j$, let $G_j\subset\Sigma_j$ be
the set supplied by Lemma~\ref{lem:second-order-expansion}.
Choose pairwise disjoint Borel sets $S_j\subset S\cap G_j$ whose
union has full $\mu$-measure, so that every $a\in S_j$ satisfies
\[
 T_aV=T_a\Sigma_j,\qquad
 0<\Theta^d(\mu,a)=\theta(a)<\infty,\qquad
 \sigma(B_r^N(a))=o(r^d),
\]
is a $\mu$-Lebesgue point of $H$ and $\Pi$, and satisfies the two
$L^1$ comparison estimates of
Lemma~\ref{lem:L1-boundary-comparison} with $\Sigma=\Sigma_j$.
The other requirements follow
by rectifiability and differentiation of measures, since
$\sigma\perp\mu$. On $S_j$ put
\[
 g=\nabla^{\Sigma_j}u=\nabla^V u,\qquad
 A=\ap\nabla_{\Sigma_j}^2u=\ap\nabla_S^2u,\qquad L=\trace A.
\]
Thus $\Pi g=g$ and $|g|\le1$. Define the distribution
\begin{equation}\label{eq:Lambda-definition}
 \Lambda(\varphi)=-\int g\cdot\nabla\varphi\dd\mu
 \qquad\bigl(\varphi\in C_c^\infty(\Omega)\bigr).
\end{equation}
The same formula is well-defined for scalar $C_c^1(\Omega)$ functions.

\begin{theorem}[Convex-Laplacian locality with generalized boundary]\label{thm:laplacian}
The distribution $\Lambda$ is a signed Radon measure. Its decomposition
relative to $\mu$ is
\begin{equation}\label{eq:Lambda-decomposition}
 \Lambda=\ell\mu+\Lambda^{\rm s},\qquad
 \ell=L+H\cdot g,\qquad
 \Lambda^{\rm s}\perp\mu,\qquad \Lambda^{\rm s}\ge-\sigma.
\end{equation}
In particular $L\in L^1_{\rm loc}(\mu)$, and for $0\le\varphi\in C_c^1(\Omega)$,
\begin{equation}\label{eq:distributional-inequality}
 \int\varphi\ell\dd\mu\le-\int g\cdot\nabla\varphi\dd\mu+\int\varphi\dd\sigma.
\end{equation}
If $\Omega=\R^N$ and $\spt\mu$ is compact, then $\Lambda$ is finite and
\begin{equation}\label{eq:integrated-ell}
 \int\ell\dd\mu\le\sigma(\R^N).
\end{equation}
\end{theorem}
\begin{proof}
\emph{Step 1: the measure property and the boundary term.}
Let $u_\varepsilon$ be smooth convex mollifications of $u$.
Then $|\nabla u_\varepsilon|\le1$ and $D^2u_\varepsilon\ge0$.
At every $x\in S_j$, restricting the supporting inequality for
$\xi\in\partial u(x)$ to $\Sigma_j$ and differentiating in opposite
tangent directions gives $\Pi_x\xi=g(x)$.
Every cluster point of $\nabla u_\varepsilon(x)$ belongs to
$\partial u(x)$, by locally uniform convergence and the supporting
inequalities for $u_\varepsilon$. Hence
\begin{equation}\label{eq:projected-mollifier}
 \Pi_x\nabla u_\varepsilon(x)\longrightarrow g(x)
 \quad\mu\text{-a.e.}
\end{equation}
For $0\le\varphi\in C_c^\infty(\Omega)$, the first variation applied to
$\varphi\nabla u_\varepsilon$ gives
\begin{align}
 -\int\Pi\nabla u_\varepsilon\cdot\nabla\varphi\dd\mu
 &=\int\varphi\,\Pi\bullet D^2u_\varepsilon\dd\mu
     +\int\varphi H\cdot\nabla u_\varepsilon\dd\mu
     -\int\varphi\eta\cdot\nabla u_\varepsilon\dd\sigma\notag\\
 &\ge-\int\varphi|H|\dd\mu-\int\varphi\dd\sigma.
 \label{eq:positive-distribution}
\end{align}
Dominated convergence on the left proves that
$\Lambda+|H|\mu+\sigma$ is a positive distribution, hence a Radon
measure. The measures $|H|\mu$ and $\sigma$ are locally finite, by
$H\in L^1_{\rm loc}(\mu)$ and the hypotheses on the first variation.
Thus $\Lambda$ is a signed Radon measure. Write its Lebesgue
decomposition provisionally as $\Lambda=\ell\mu+\Lambda^{\rm s}$.
Since $\sigma\perp\mu$, the part singular with respect to $\mu$ of
the preceding positive measure is $\Lambda^{\rm s}+\sigma$.
Consequently $\Lambda^{\rm s}\ge-\sigma$.
By $C^1$ approximation, its action on scalar $C_c^1$ functions is still
\eqref{eq:Lambda-definition}.

Refine the sets $S_j$ by deleting a further $\mu$-null set, so that
every $a\in S_j$ is a $\mu$-Lebesgue point of $\ell$ and
$|\Lambda^{\rm s}|(B_r^N(a))=o(r^d)$.
This is possible by differentiation of measures and the positive finite
$d$-density already imposed on $S_j$.

\medskip\noindent\emph{Step 2: a local comparison field.}
Fix $\Sigma=\Sigma_i$ and $a\in S_i$. Let $\widehat u=\widehat u_a$ be the
$C^2$ function in Lemma~\ref{lem:second-order-expansion}.
On a sufficiently small tubular neighborhood $U\Subset\Omega$ of $a$,
with $\np_\Sigma(U)\subset\operatorname{dom}\widehat u$, define
\begin{equation}\label{eq:X-extension}
 X=\nabla^\Sigma\widehat u\circ\np_\Sigma.
\end{equation}
Then $X\in C^1(U;\R^N)$.
The prescribed first and second derivatives of $\widehat u$ and
$\Der\np_\Sigma(a)=\Pi_a$ imply
\begin{equation}\label{eq:X-jet}
 X(a)=g(a),\qquad
 \Pi_a\bullet\Der X(a)
 =\trace\nabla_\Sigma^2\widehat u(a)=L(a).
\end{equation}
We claim that
\begin{equation}\label{eq:g-strong-jet}
 \int_{B_r^N(a)}|g-X|\dd\mu=o(r^{d+1})\qquad(r\downarrow0).
\end{equation}

Fix $0<r_a<r_{a,3}$ so small that
$\overline{B_{4r_a}^N(a)}\subset U$ and
$\overline{B_{4r_a}^N(a)}\cap\Sigma$ is a compact subset of
$\operatorname{dom}\widehat u$, where $r_{a,3}$ is supplied by
Lemma~\ref{lem:second-order-expansion}. Thus
\eqref{eq:second-order-expansion tilde u} holds with $c=3$ for
$0<r<r_a$.
Moreover, we observe that there are constants $0< C_a<\infty$ and  $ 0< \delta_a <1 $ such that for every 
$y\in B_{2r_a}^N(a)\cap\Sigma$ the maps
\[
 f_y:B_{\delta_a}^{T_y\Sigma}(0)\longrightarrow T_y^\perp\Sigma,
 \qquad f_y(0)=0,\quad \Der f_y(0)=0,
\]
parametrize $\Sigma$ by $p\mapsto y+p+f_y(p)$ and satisfy
\[
 |f_y(p)|\le C_a|p|,\qquad
 \|\Der f_y(p)\|\le C_a,\qquad
 \|\Der^2 f_y(p)\|\le C_a
 \quad (|p|<\delta_a).
\]
Choose $ 0 <\tau_a < 1 $ so that  $\tau_a r_a<\delta_a$ and
$(1+C_a)\tau_a\le\tfrac12$. Then
\[
 \{y+p+f_y(p):|p|<\tau_a r\}
 \subset\Sigma\cap B_{3r}^N(a) \quad \text{for $ y \in \Sigma\cap B^N_{2r}(a) $ and $ 0 < r < r_a $\,.}
\]
On the smaller domains $|p|<\tau_a r_a$, the scalar functions
$p\mapsto\widehat u(y+p+f_y(p))$ also have uniformly bounded second
derivatives for $y\in\Sigma\cap B_{2r_a}^N(a)$.
All these constants are fixed independently of $r,j,z,e,\tau$.

Choose $0<r< r_a$, $j\ge1$, and $z\in B_r^N(a)\cap S_j$. Then put
\[
 y=\np_\Sigma(z)\qquad \text{and} \qquad  w=z-y=\nu_\Sigma(z).
\]
Since $a\in\Sigma$, $|w|\le|z-a|$ and $|y-a|<2r$.
Lemma~\ref{lem:second-order-expansion}, with $c=3$, gives
$\omega=\omega_{a,3}$ such that
\begin{equation}\label{eq:second-order-uniform}
 |u(q)-\widehat u(q)|\le\omega(r)r^2
 \qquad(q\in\Sigma\cap B_{3r}^N(a)).
\end{equation}
For $e\in T_y\Sigma$ with $|e|=1$ and $0<\tau<\tau_a$, set
\[
s=\tau r \quad \text{and} \quad   y_\pm=y\pm se+f_y(\pm se)\in\Sigma\cap B_{3r}^N(a).
\]
The aforementioned uniform second-derivative bounds then give a
constant $C'_a<\infty$ such that
\begin{align*}
 |y_+-y - se| \le C'_as^2, \qquad |y_--y + se| \le C'_as^2\\
| \widehat u(y_+)-\widehat u(y) - sX(z)\cdot e | \le C'_as^2, \qquad | \widehat u(y_-)-\widehat u(y) + sX(z)\cdot e | \le C'_as^2\,.
\end{align*}
Write
\[
 R=u|_\Sigma-\widehat u,\qquad
 R_\pm=y_\pm-y\mp se,\qquad
 E_\pm=\widehat u(y_\pm)-\widehat u(y)\mp sX(z)\cdot e.
\]
Thus $|R_\pm|\le C'_as^2$ and $|E_\pm|\le C'_as^2$.
By \eqref{eq:second-order-uniform},
$|R(y_\pm)-R(y)|\le2\omega(r)r^2$. For every $\xi\in\partial u(z)$, the first-order properties on $S_j$ give
\begin{equation}\label{eq:tangential-subgradient}
	|\xi|\le1,\qquad g(z)=\Pi^V_z\xi.
\end{equation}
The supporting inequality at $z=y+w$ and $ \Lip(u) \leq 1 $ yields
\[
 \xi\cdot(y_\pm-y)
 \le u(y_\pm)-u(z)+\xi\cdot w
 \le u(y_\pm)-u(y)+2|w|.
\]
The two Taylor formulas give the exact identities
\begin{align*}
 \xi\cdot(y_\pm-y)
 &=\pm s\,\xi\cdot e+\xi\cdot R_\pm,\\
 u(y_\pm)-u(y)
 &=\widehat u(y_\pm)-\widehat u(y)+R(y_\pm)-R(y)\\
 &=\pm s\,X(z)\cdot e+E_\pm+R(y_\pm)-R(y).
\end{align*}
Substituting into the preceding inequality and rearranging gives,
for each sign,
\[
 \pm s\,(\xi-X(z))\cdot e
 \le E_\pm-\xi\cdot R_\pm+R(y_\pm)-R(y)+2|w|.
\]
Since $e\in T_y\Sigma$, one has
$\xi\cdot e=(\Pi_y^\Sigma\xi)\cdot e$. Using $|\xi|\le1$ and the
three remainder bounds, we obtain
\begin{align*}
 \pm s(\Pi_y^\Sigma\xi-X(z))\cdot e
 &\le |E_\pm|+|\xi|\,|R_\pm|+|R(y_\pm)-R(y)|+2|w|\\
 &\le 2C'_as^2+2\omega(r)r^2+2|w|.
\end{align*}
The coefficient $2C'_a$ accounts for the two separate quadratic
remainders $E_\pm$ and $\xi\cdot R_\pm$.
Divide by $s$, use both signs, and take the supremum over unit
$e\in T_y\Sigma$. Since $\Pi^\Sigma_y\xi$ and $X(z)$ are tangent at
$y$, \eqref{eq:tangential-subgradient} implies
\begin{equation}\label{eq:gradient-comparison}
 |g(z)-X(z)|\le|\Pi^V_z-\Pi^\Sigma_y|
   +2C'_a\tau r+\frac{2\omega(r)r}{\tau}
   +\frac{2|\nu_\Sigma(z)|}{\tau r}.
\end{equation}
All constants are independent of $j$, so the estimate \eqref{eq:gradient-comparison} holds for
$\mu$-almost every $z\in B_r^N(a)$.
Choose $D_a<\infty$ with $\mu(B_r^N(a))\le D_ar^d$ for small $r$.
Integration and \eqref{eq:boundary-L1-tilt} give
\[
 \frac1{r^{d+1}}\int_{B_r^N(a)}|g-X|\dd\mu
 \le o(1)+2D_aC'_a\tau+\frac{2D_a\omega(r)}{\tau}
   +\frac{2}{\tau r^{d+2}}\int_{B_r^N(a)}|\nu_\Sigma|\dd\mu.
\]
For each fixed $\tau\in(0,\tau_a)$, the last term tends to zero by
\eqref{eq:boundary-L1-height}. First let $r\downarrow0$, and then
$\tau\downarrow0$, to prove \eqref{eq:g-strong-jet}.

\Needspace{8\baselineskip}
\medskip\noindent\emph{Step 3: identification of the density.}
Let $a\in S_i$ be as in Step~2; all differentiation properties used
below hold by the choice of $S_i$ and its refinement in Step~1.
Take $0\le\varphi\in C_c^1(B_1^N(0))$ with
$\varphi=1$ on $B_{1/2}^N(0)$ and set
$\varphi_r(z)=\varphi((z-a)/r)$.
Since $\Pi g=g$,
$|g-\Pi X|=|\Pi(g-X)|\le|g-X|$.
Equation~\eqref{eq:g-strong-jet} and the first variation applied to
$\varphi_rX$, extended by zero, yield
\begin{align}
 \Lambda(\varphi_r)
 &=-\int\Pi X\cdot\nabla\varphi_r\dd\mu+o(r^d)\notag\\
 &=\int\varphi_r(\Pi\bullet\Der X+H\cdot X)\dd\mu
   -\int\varphi_r\eta\cdot X\dd\sigma+o(r^d).
 \label{eq:blowup-Lambda}
\end{align}
Notice $X(z)$ need not lie in $T_zV$. The additional boundary term
is negligible: $X$ is bounded near $a$ and the choice of $S_i$ gives
\[
 \left|\int\varphi_r\eta\cdot X\dd\sigma\right|
 \le\|\varphi\|_\infty\sup_{B_r^N(a)}|X|\,
       \sigma(B_r^N(a))=o(r^d).
\]
On the left, differentiation of $\ell\mu$ and the vanishing
$d$-density of $|\Lambda^{\rm s}|$ give
$\Lambda(\varphi_r)=\ell(a)\int\varphi_r\dd\mu+o(r^d)$.
Since $a$ is a $\mu$-Lebesgue point of the locally integrable field $H$,
\[
 \int_{B_r^N(a)}|H(z)-H(a)|\dd\mu(z)=o(r^d),\qquad
 \int_{B_r^N(a)}|H|\dd\mu=O(r^d).
\]
On the right, these estimates, continuity of $X,\Der X$, and the
$\mu$-Lebesgue value of $\Pi$ give, by \eqref{eq:X-jet},
\[
 \int\varphi_r(\Pi\bullet\Der X+H\cdot X)\dd\mu
 =\bigl(L(a)+H(a)\cdot g(a)\bigr)\int\varphi_r\dd\mu+o(r^d).
\]
Since
$\int\varphi_r\dd\mu\ge\mu(B_{r/2}^N(a))\ge c(a)r^d$
for small $r$, it follows that $\ell(a)=L(a)+H(a)\cdot g(a)$.
The sets $S_i$ cover $\mu$-almost all points, proving the density
identity. It also gives $L\in L^1_{\rm loc}(\mu)$, since
$\ell,H\in L^1_{\rm loc}(\mu)$ and $|g|\le1$.

The bound $\Lambda^{\rm s}\ge-\sigma$ now proves
\eqref{eq:distributional-inequality}.
If $\Omega=\R^N$ and $\spt\mu$ is compact, then $\Lambda$ is finite
and supported there. Testing with a function equal to one near
$\spt\mu$ gives $\Lambda(\R^N)=0$ and hence
$\int\ell\dd\mu=-\Lambda^{\rm s}(\R^N)\le\sigma(\R^N)$.
In this compactly supported case $\sigma$ is also supported on
$\spt\mu$, since $\|\delta V\|=|H|\mu+\sigma$, and hence has finite mass.
\end{proof}

\begin{remark}\label{rem:drift}
For integral varifolds, $H\perp T_xV$ almost everywhere, so $\ell=L$;
see \cite[Theorem 1 and Remark 3.7]{Menne13}.
For a smooth submanifold $\Sigma$ with smooth positive density,
\[
 H=H_\Sigma+\nabla^\Sigma\log\theta,\qquad
 \ell=L+\nabla^\Sigma\log\theta\cdot\nabla^\Sigma u.
\]
Thus the term $H\cdot g$ cannot be omitted for real multiplicities.
The proof above does not assume perpendicularity of $H$.
\end{remark}

\begin{corollary}[Normal mean curvature with generalized boundary]
\label{cor:normal-locality-boundary}
Under the hypotheses of Theorem~\ref{thm:laplacian}, for
$\mu$-almost every $x\in S$,
\[
 H_S(x,\omega)=H(x)\cdot\omega
 \qquad\text{for every }\omega\in T_x^\perp S.
\]
\end{corollary}
\begin{proof}
Fix a constant unit vector $e\in\R^N$ and apply
Theorem~\ref{thm:laplacian} to the affine function $u_e(x)=x\cdot e$.
Its tangential gradient is $g_e=\Pi e$. The first variation applied
to $\varphi e$ gives
\[
 \Lambda_e(\varphi)=-\delta V(\varphi e)
 =\int\varphi H\cdot e\dd\mu-\int\varphi\eta\cdot e\dd\sigma.
\]
Since $\sigma\perp\mu$, comparison of the absolutely continuous
densities with Theorem~\ref{thm:laplacian} yields
\[
 L_e:=\trace(\ap\nabla_S^2u_e)
 =H\cdot e-H\cdot\Pi e=H\cdot(I-\Pi)e.
\]
The affine-function identity in Section~\ref{sec:semiconvex} gives
$L_e(x)=H_S(x,(I-\Pi_x)e)$.
Apply this identity to a fixed orthonormal basis $e_1,\ldots,e_N$
and discard the union of the finitely many exceptional null sets.
At every remaining point, the vectors $(I-\Pi_x)e_\alpha$ span
$T_x^\perp S$; linearity in the normal vector proves the assertion.
\end{proof}

\section{The isoperimetric inequality}\label{sec:isoperimetric}

For $2\le d<N$, put $q=N-d$ and define
\begin{equation}\label{eq:Brendle-constant}
 \Ciso=
 \begin{cases}
 d\,\omega_d^{1/d},&q=1,\\[1mm]
 \displaystyle d\left(\frac{N\omega_N}{q\omega_q}\right)^{1/d},&q\ge2.
 \end{cases}
\end{equation}
For $q=2$ the second expression is also $d\omega_d^{1/d}$, because
$(d+2)\omega_{d+2}=2\omega_2\omega_d$.
These are the constants of \cite{Brendle2021,BrendleEichmairTransport}.
Sharpness for $q=1,2$ is proved in Corollary~\ref{cor:sharp-isoperimetric};
no optimality assertion is made for $q\ge3$.

Let $\Omega\subset\R^N$ be open and let $V=\mathbf v(S,\theta)$ be a
rectifiable $d$-varifold with
$\mu=\|V\|=\theta\cH^d\restr S$, $\theta\ge\theta_0>0$
$\mu$-almost everywhere. Assume that there are $d$-dimensional $C^2$
submanifolds $\Sigma_j\subset\Omega$ such that
$\mu(S\setminus\bigcup_j\Sigma_j)=0$, and that
\eqref{eq:first-variation-boundary} holds, where
$H\in L^1_{\rm loc}(\mu;\R^N)$, $\sigma$ is a positive Radon
measure singular with respect to $\mu$, and $\eta$ is Borel with
$|\eta|=1$ $\sigma$-almost everywhere. No condition is imposed on the
support of $\sigma$.
The total variation is then $\|\delta V\|=|H|\mu+\sigma$.
In particular $\sigma$ is the singular part of $\|\delta V\|$ relative
to $\mu$. The case without generalized boundary is obtained by
taking $\sigma=0$.

\begin{theorem}[Isoperimetric inequality with generalized boundary]
\label{thm:isoperimetric}
Let $\Omega=\R^N$ and suppose the preceding hypotheses hold.
Assume in addition that
\[
 0<M:=\mu(\R^N)<\infty,\qquad
 H\in L^1(\mu;\R^N),\qquad \sigma(\R^N)<\infty.
\]
Then $\theta^{1/d}\in L^1(\mu)$ and
\begin{equation}\label{eq:iso-refined}
 \sigma(\R^N)+\int|H|\dd\mu
 \ge\Ciso M^{-1/d}\int\theta^{1/d}\dd\mu
 \ge\Ciso\theta_0^{1/d}M^{(d-1)/d}.
\end{equation}
\end{theorem}

We first record the radial densities used in the proof. Each density
below is defined on all of $\R^N$, with value zero on and outside the
unit sphere.

\begin{lemma}[Normal slices of radial source densities]\label{lem:slices}
For a nonnegative radial Borel density $\rho$ with
$\int_{\R^N}\rho\dd\cL^N=1$ and support in $B^N$, set
\begin{equation}\label{eq:slice-constant}
 a_\rho=\sup_{z\in\R^d}\int_{\R^q}\rho(z,\omega)\dd\omega.
\end{equation}
If $q\ge2$ and $0<t<1$, then
\begin{equation}\label{eq:annular-density}
 \rho_t(\xi)=\frac{1_{\{t<|\xi|<1\}}}{\omega_N(1-t^N)}
\end{equation}
satisfies
\begin{equation}\label{eq:annular-slice}
 a_{\rho_t}=\frac{\omega_q(1-t^q)}{\omega_N(1-t^N)}
 \longrightarrow\frac{q\omega_q}{N\omega_N}
 \qquad(t\uparrow1).
\end{equation}
If $q=1$, the density
\begin{equation}\label{eq:codim-one-density}
 \rho(\xi)=
 \begin{cases}
 (\pi\omega_d\sqrt{1-|\xi|^2})^{-1},&|\xi|<1,\\
 0,&|\xi|\ge1,
 \end{cases}
\end{equation}
has integral one and $a_\rho=\omega_d^{-1}$.
\end{lemma}
\begin{proof}
We use the projection formula for beta densities in
\cite[Lemma~4.3(a)]{KabluchkoTemesvariThale2019}, together with the
normalization in equation~(1.1) of that paper. If
$f_s(\xi)=(\omega_Ns^N)^{-1}1_{\{|\xi|<s\}}$, then
\[
 \rho_t=\frac{f_1-t^Nf_t}{1-t^N}
\]
up to values on the bounding spheres, which do not affect any normal
slice. Apply the projection formula with beta parameter $0$ to $f_1$
and $f_t$, and take this weighted difference of their marginals.
For $q\ge2$ the resulting radial slice density is maximal at the
origin, giving \eqref{eq:annular-slice}. The displayed limit follows from
$(1-t^q)/(1-t^N)\to q/N$.  When $q=1$, the same projection formula with
beta parameter $-1/2$ gives exactly the density
\eqref{eq:codim-one-density} and shows that its $d$-dimensional marginal
is the uniform probability density $\omega_d^{-1}$ on $B^d$; hence
$a_\rho=\omega_d^{-1}$ and $\int\rho\dd\cL^N=1$.  See also
\cite[Section~3]{BrendleEichmairTransport} for the annular limiting
construction used in the optimal-transport proof.
\end{proof}

\begin{proof}[Proof of Theorem~\ref{thm:isoperimetric}]
\emph{Step 1: the Laplacian and normal mean curvature.}
The global assumption $H\in L^1(\mu)$ implies the local integrability
required by Lemma~\ref{lem:L1-boundary-comparison} and
Theorem~\ref{thm:laplacian}. For any convex $1$-Lipschitz
$u:\R^N\to\R$, use the notation
$g,A,L,\Lambda$ of Theorem~\ref{thm:laplacian}. That theorem gives
$\Lambda=\ell\mu+\Lambda^{\rm s}$, $\ell=L+H\cdot g$,
$\Lambda^{\rm s}\ge-\sigma$, and the inequality
\eqref{eq:distributional-inequality}.
Corollary~\ref{cor:normal-locality-boundary} gives
\begin{equation}\label{eq:boundary-Laplacian-decomposition}
 H_S(x,\omega)=H(x)\cdot\omega
 \qquad(\omega\in T_x^\perp S),\quad\mu\text{-a.e. }x.
\end{equation}

\medskip\noindent\emph{Step 2: the pointwise transport estimate.}
Now take $\Omega=\R^N$.
Choose a radial $\chi\in C_c^\infty(\R^N)$, nonincreasing in the radius,
with $0\le\chi\le1$, $\chi=1$ on $B_1^N(0)$, and
$\spt\chi\subset B_2^N(0)$. Set
\[
 \chi_R(x)=\chi(x/R),\qquad
 M_R=\int\chi_R\dd\mu,\qquad
 \alpha_R=M_R^{-1}\chi_R\mu.
\]
For all sufficiently large $R$, $M_R>0$;
$\chi_R\uparrow1$, $M_R\uparrow M$, and
$|\nabla\chi_R|\le C_\chi/R$.
The probability $\alpha_R$ has compact support
$\Gamma_R=\spt\alpha_R$.

Fix one radial density $\rho$ from Lemma~\ref{lem:slices}, with
$0<a_\rho<\infty$, and apply the transport construction to
$\beta=\rho\cL^N$ and $\alpha_R$.
Let $u_R$ be its ambient convex $1$-Lipschitz potential, and let
$g_R,A_R,L_R,\ell_R$ be the quantities given by Step~1 for the
\emph{original} varifold $V$.
For the transport target use
\[
 S_R=S\cap\Gamma_R\cap\{\chi_R>0\},\qquad
 \theta_R=\theta\chi_R.
\]
The complement of $\Gamma_R$ in $S\cap\{\chi_R>0\}$ is $\mu$-null,
and $\alpha_R=M_R^{-1}\theta_R\cH^d\restr S_R$.
The same $\Sigma_j$ and their derivatives may be used on $S_R$;
Section~3 requires no positive lower bound for $\theta_R$.
All first-variation and Laplacian assertions remain assertions for $V$,
not for the reweighted varifold $\chi_RV$.

Let $\cY_{R,x}$ and $\Phi_R(x,\omega)=g_R(x)+\omega$ denote the exact
active fibers and the shifted map. Lemma~\ref{lem:fiber-balance} gives,
for $\mu$-almost every $x$ with $\chi_R(x)>0$,
\begin{equation}\label{eq:localized-fiber-balance}
 \frac{\theta(x)\chi_R(x)}{M_R}
 =\int_{\cY_{R,x}}\det\mathcal B_R(x,\omega)
      \rho(\Phi_R(x,\omega))\dd\cH^q(\omega),
 \qquad \mathcal B_R=A_R-\cW_{(x,\omega)}S.
\end{equation}
On these fibers, $\mathcal B_R\ge0$ and $|\Phi_R|\le1$.
All excluded base sets have zero target mass and hence
zero active $\beta$-image by Lemma~\ref{lem:transport-map}.
This remains true for annular densities, which need not be positive
everywhere on the ball.

Using \eqref{eq:boundary-Laplacian-decomposition} and $\ell_R=L_R+H\cdot g_R$, at every active
normal vector over such a point,
\begin{align}
 \trace\mathcal B_R(x,\omega)
 &=L_R(x)-H_S(x,\omega)\notag\\
 &=L_R(x)-H(x)\cdot\omega\notag\\
 &=\ell_R(x)-H(x)\cdot\Phi_R(x,\omega)
 \le\ell_R(x)+|H(x)|.
 \label{eq:contact-trace}
\end{align}
The positive left side of \eqref{eq:localized-fiber-balance} implies
that some active matrix has positive determinant. Its trace is
positive, so $\ell_R(x)+|H(x)|>0$.
The arithmetic--geometric mean inequality for the nonnegative
eigenvalues and the slice bound give
\begin{align}
 \frac{\theta(x)\chi_R(x)}{M_R}
 &\le\left(\frac{\ell_R(x)+|H(x)|}{d}\right)^d
       \int_{\cY_{R,x}}\rho(g_R(x)+\omega)\dd\cH^q(\omega)\notag\\
 &\le a_\rho\left(\frac{\ell_R(x)+|H(x)|}{d}\right)^d.
 \label{eq:pointwise-determinant}
\end{align}
Indeed $g_R(x)\perp\omega$, and radiality identifies the integral over
the whole normal space with a slice in \eqref{eq:slice-constant}.
Taking the positive $d$th root yields
\begin{equation}\label{eq:pointwise-root}
 \ell_R+|H|\ge d\,a_\rho^{-1/d}
       \left(\frac{\theta\chi_R}{M_R}\right)^{1/d}
 \quad\mu\text{-a.e. on }\{\chi_R>0\}.
\end{equation}
The exceptional sets may depend on $R$ and $\rho$; only the integrated
inequalities for each fixed choice will be used.

\medskip\noindent\emph{Step 3: integration and passage to finite mass.}
Multiply \eqref{eq:pointwise-root} by $\chi_R$ and integrate.
Its nonnegative right side is integrable because the left side has a
locally integrable absolute value on $\spt\chi_R$.
Equation~\eqref{eq:distributional-inequality}, with $\varphi=\chi_R$, gives
\begin{align}
 d\,a_\rho^{-1/d}M_R^{-1/d}
       \int\theta^{1/d}\chi_R^{1+1/d}\dd\mu
 &\le\int\chi_R(\ell_R+|H|)\dd\mu\notag\\
 &\le-\int g_R\cdot\nabla\chi_R\dd\mu
       +\int\chi_R|H|\dd\mu+\int\chi_R\dd\sigma\notag\\
 &\le\frac{C_\chi M}{R}
       +\int\chi_R|H|\dd\mu+\int\chi_R\dd\sigma.
 \label{eq:isoperimetric-cutoff}
\end{align}
The cutoff error uses only $|g_R|\le1$ and is uniform in the potential.
As $R\to\infty$, monotone convergence applies to
$\int\theta^{1/d}\chi_R^{1+1/d}\dd\mu$, initially allowing value
$+\infty$, and $M_R\to M>0$.
On the right, the limit is finite by $H\in L^1(\mu)$ and finiteness of
$\sigma$. Thus
\begin{equation}\label{eq:isoperimetric-source-bound}
 \sigma(\R^N)+\int|H|\dd\mu
 \ge d\,a_\rho^{-1/d}M^{-1/d}\int\theta^{1/d}\dd\mu,
\end{equation}
and this also proves $\theta^{1/d}\in L^1(\mu)$.
No convergence of $u_R$ or of its derivatives is required.
When $\spt\mu$ is compact, the same estimate follows directly by
transporting to $\mu/M$ and testing \eqref{eq:distributional-inequality}
with a cutoff equal to one near $\spt\mu$.

\medskip\noindent\emph{Step 4: the constant.}
For $q\ge2$, apply \eqref{eq:isoperimetric-source-bound} to each
$\rho_t$ and let $t\uparrow1$. Lemma~\ref{lem:slices} gives
$d(N\omega_N/(q\omega_q))^{1/d}$.
For $q=1$, use \eqref{eq:codim-one-density} to get $d\omega_d^{1/d}$. Finally,
$\int\theta^{1/d}\dd\mu\ge\theta_0^{1/d}M$ proves
\eqref{eq:iso-refined}.
\end{proof}

\begin{corollary}[Sharpness in codimensions one and two]
\label{cor:sharp-isoperimetric}
Under the hypotheses of Theorem~\ref{thm:isoperimetric}, if
$N-d\in\{1,2\}$, then
\begin{equation}\label{eq:sharp-isoperimetric}
 \sigma(\R^N)+\int|H|\dd\mu
 \ge d\omega_d^{1/d}M^{-1/d}\int\theta^{1/d}\dd\mu
 \ge d\omega_d^{1/d}\theta_0^{1/d}M^{(d-1)/d}.
\end{equation}
The coefficient $d\omega_d^{1/d}$ is optimal, for every fixed
$\theta_0>0$. Equality in both inequalities is attained by a flat round
$d$-ball with constant multiplicity $\theta_0$.
For integral varifolds with $\theta_0=1$, equality is attained by a
multiplicity-one flat round ball. Thus optimality already holds with
$H=0$ and a smooth $(d-1)$-dimensional generalized boundary.
\end{corollary}
\begin{proof}
The values of \eqref{eq:Brendle-constant} coincide in these two
codimensions. Let $P\subset\R^N$ be an affine $d$-plane, $a\in P$,
$s>0$, and $D=P\cap B_s^N(a)$. For $V=\mathbf v(D,\theta_0)$,
the divergence theorem in $P$ gives
\[
 H=0,\qquad
 \sigma=\theta_0\cH^{d-1}\restr\partial_P D,\qquad
 \eta(x)=(x-a)/s\quad(x\in\partial_P D).
\]
The set $B=\partial_PD$ is closed and smooth and satisfies $\mu(B)=0$.
Furthermore,
\[
 M=\theta_0\omega_ds^d,\qquad
 \sigma(\R^N)=\theta_0d\omega_ds^{d-1},\qquad
 \int\theta^{1/d}\dd\mu=\theta_0^{1/d}M.
\]
Substitution gives equality throughout
\eqref{eq:sharp-isoperimetric}; taking $\theta_0=1$ proves optimality
already in the integral class.
\end{proof}

\begin{corollary}[Integral varifolds with generalized boundary]
\label{cor:integral-iso}
Let $V$ be a nonzero integral $d$-varifold in $\R^N$, $2\le d<N$,
with finite mass and finite total first variation:
\[
 0<M:=\|V\|(\R^N)<\infty,\qquad \|\delta V\|(\R^N)<\infty.
\]
Put $\mu=\|V\|$ and $\theta=\Theta^d(\mu,\cdot)$ almost everywhere,
and write the Lebesgue decomposition of $\delta V$ as in
\eqref{eq:first-variation-boundary}. Then $\theta^{1/d}\in L^1(\mu)$ and
\[
 \begin{aligned}
 \|\delta V\|(\R^N)
 &=\sigma(\R^N)+\int|H|\dd\mu\\
 &\ge\Ciso M^{-1/d}\int\theta^{1/d}\dd\mu
 \ge\Ciso M^{(d-1)/d}.
 \end{aligned}
\]
\end{corollary}
\begin{proof}
The Lebesgue decomposition of the finite vector-valued measure
$\delta V$ gives $H\in L^1(\mu)$, $\sigma\perp\mu$,
$|\eta|=1$ $\sigma$-almost everywhere, and
\[
 \int|H|\dd\mu+\sigma(\R^N)=\|\delta V\|(\R^N)<\infty.
\]
Integrality gives $\theta\ge1$ almost everywhere.
Since $V$ has locally bounded first variation,
\cite[Theorem 1]{Menne13} supplies $C^2$ submanifolds $\Sigma_j$
whose union has full $\mu$-measure. 
All hypotheses of Theorem~\ref{thm:isoperimetric} therefore hold with
$\theta_0=1$. Its conclusion proves the displayed inequalities and
integrability of $\theta^{1/d}$.
\end{proof}

\textbf{Ackowledgements.} AI tools (Open AI, GPT 6) were used during the preparation of this manuscript.

\end{document}